\documentclass{amsart}

\usepackage{graphicx,amssymb,mathrsfs,amsmath,color,fancyhdr}
\usepackage[all]{xy}
\newtheorem{theorem}{Theorem}[section]
\newtheorem{lemma}[theorem]{Lemma}

\newtheorem{Conjecture}[theorem]{Conjecture}

\theoremstyle{definition}
\newtheorem{definition}[theorem]{Definition}

\newtheorem{corollary}[theorem]{Corollary}

\theoremstyle{remark}

\numberwithin{equation}{section}

\newcommand{\bdot}{\boldsymbol{\cdot}}

\begin{document}
	
	\title[On some zero-sum invariants for abelian groups of rank three]
	{On some zero-sum invariants for abelian groups of rank three}
	
	\begin{abstract}
		Let $G$ be an additive finite abelian
		group with exponent $\exp(G)$.
		For $L\subseteq \mathbb N$, let $\mathsf{s}_{L}(G)$ be the smallest integer $\ell$ such that every sequence $S$ over $G$ of length $\ell$ has a zero-sum subsequence $T$ of length $|T|\in L$.
		In this paper, we consider the invariants $\mathsf{s}_{[1,t]}(G)$ and $\mathsf{s}_{\{k\exp(G)\}}(G)$ (with $k\in \mathbb N$). We obtain precise values as well as upper bounds of the above invariants for some abelian groups of rank three. Some of these results improve previous results of Gao-Thangadurai and Han-Zhang.

		%Let $\eta(G)$ be the smallest integer $t$ such that every sequence of length $t$ has a nonempty zero-sum subsequence of length at most $\exp(G)$. Let $\mathsf{s}(G)$ be the EGZ-constant of $G$, which is defined as the  smallest integer $t$ such that every sequence of length $t$ has a zero-sum subsequence of length $\exp(G)$. Let $p$ be an odd prime.  We determine $\eta(G)$ for some groups $G$ with  $\mathsf{D}(G)\leq 2\exp(G)-1$, including the $p$-groups of rank three and the $p$-groups $G=C_{\exp(G)}\oplus C_{p^m}^r$. We also determine $\mathsf{s}(G)$ for the groups $G$ above with more larger exponent than $\mathsf D(G)$, which confirms a conjecture by Schmid and Zhuang from 2010, where $\mathsf D(G)$ denotes the Davenport constant of $G$.
	\end{abstract}
	
	\author{Shiwen Zhang}
	\address{School of Mathematics (Zhuhai), Sun Yat-sen University, Zhuhai 519082, Guangdong, P.R. China}
	\email{zhangshw9@mail2.sysu.edu.cn}
	
	\thanks{}

	\maketitle
	\section{Introduction}

	Let $G$ be an additive finite abelian group with exponent $\exp(G)$. Let $S=g_1\bdot\ldots\bdot g_\ell$ be a sequence over $G$ (unordered and repetition is allowed), where $g_i\in G$ for $1\le i\le \ell$. We denote by $|S|:=\ell$, which is called the length of the sequence $S$. We call $S$ a zero-sum sequence if $\sigma(S):=\sum^\ell_{i=1}g_i=0$. The essential idea of the direct zero-sum theory is that a sequence $S$ with enough elements will contain a zero-sum subsequence with prescribed properties. For example, in 1961, Erd\H{o}s, Ginzburg and Ziv \cite{EGZ} proved that from every sequence $S$ over an abelian group of order $n$ of length $2n-1$, we can always find a zero-sum subsequence $T$ of length $n$ (see \cite{ADZ} for other different proofs).
	For $L\subseteq \mathbb N$, let $\mathsf{s}_{L}(G)$ be the smallest integer $\ell$ (if there exists) such that every sequence $S$ over $G$ of length $\ell$ has a zero-sum subsequence $T$ of length $|T|\in L$.
	\begin{definition}
		We denote
		\begin{itemize}
			\item $\mathsf D(G):=\mathsf{s}_{\mathbb N}(G)$, which is called the Davenport constant of $G$;
			\item $\mathsf{s}_{k\exp(G)}(G):=\mathsf{s}_{\{k\exp(G)\}}(G)$ (with $k\in \mathbb N$), which is called the $k$-th Erd\H{o}s-Ginzburg-Ziv constant of $G$;
			\item $\mathsf{s}_{\le t}(G):=\mathsf{s}_{[1,t]}(G)$ (for some $t$ with $\exp(G)\le t\le \mathsf D(G)$). In particular, we denote $\eta(G):=\mathsf{s}_{\le \exp(G)}(G)$, which is called the $\eta$-constant of $G$.
		\end{itemize}
	\end{definition}
	
	The above invariants have received a lot of attention,  we refer to \cite{[GaoG]} for a survey of zero-sum theory. We shall focus on $\mathsf{s}_{k\exp(G)}(G)$ and $\mathsf{s}_{\le t}(G)$ in this paper. When $k=1$, $\mathsf s(G):=\mathsf s_{\exp(G)}(G)$ is the famous Erd\H{o}s-Ginzburg-Ziv constant. So far, roughly speaking, precise values of $\mathsf s(G)$ have been obtained only for groups of rank at most two and few groups of higher rank; see, e.g., \cite{AD,EEGKR,EllG,FGZ,FanZ,FoxS,BGS,Hege,Kem,Lisa,LisaZa,Nas,[Rei],Zakh} (in particular, $\mathsf s(C_3^r)$ is related to the famous cap-set problem). 
	
	It is easy to verify that $\mathsf s_{k\exp(G)}(G)\ge k\exp(G)+\mathsf D(G)-1$ holds for every $k\ge1$; see \cite{Gao1}. In 1996, Gao \cite{Gao3} proved that
	$\mathsf s_{k\exp(G)}(G)= k\exp(G)+\mathsf D(G)-1$,
	provided that $k\exp(G)\ge |G|$. In 2006, Gao and Thangadurai \cite{GaoRT} showed that if $k\exp(G)<\mathsf D(G)$ then $\mathsf s_{k\exp(G)}(G)> k\exp(G)+\mathsf D(G)-1$. Recently, Gao, Han, Peng and Sun \cite{GHPS}
	proposed the following conjecture.
	\begin{Conjecture}\label{GHPSC}
		Let $G$ be a finite abelian group. If $k\exp(G)\ge \mathsf D(G)$, then we have
		$\mathsf s_{k\exp(G)}(G)= k\exp(G)+\mathsf D(G)-1$.
	\end{Conjecture}
	Note that, for groups of the form $C_n^r$, the precise values of their Davenport constant are unknown (with the conjecture $\mathsf D(C_n^r)=r(n-1)+1$). In this case, Kubertin \cite{K} conjectured that
	$\mathsf s_{kn}(C_n^r)=(k+r)n-r$. Conjecture \ref{GHPSC} has been verified for abelian $p$-groups $G$ with $\mathsf D(G)\le 4\exp(G)$ with the restriction that $p\ge 5$ (very recently, this result (also for $p\ge 5$ in the rank $3$ case) was reproved with a new approach by Grynkiewicz \cite{DGr}). Now, we focus on the cases $p=2$ or $3$.
	Gao and Thangadurai \cite{GaoRT} proved that
	\begin{itemize}
		\item $\mathsf{s}_{k\cdot 2}(C_{2}^3)=2k+3$, where $k\ge 2$;
		\item $\mathsf{s}_{2\cdot 3}(C_{3}^3)=13$, $15\le\mathsf{s}_{3\cdot 3}(C_{3}^3)\le 17$, $\mathsf{s}_{k\cdot 3}(C_{3}^3)=3k+6$, where $k\ge 4$.
	\end{itemize}
	Moreover, in \cite{HZ2}, Han and Zhang proved that
	\begin{itemize}
		\item $\mathsf{s}_{k\cdot 2^n}(C_{2^n}^3)=(k+3)2^n-3$, where $k\ge 4$,
		\item $\mathsf{s}_{k\cdot 3^n}(C_{3^n}^3)=(k+3)3^n-3$, where $k\ge 6$.
	\end{itemize}
	Consequently, they obtained the following asymptotically tight bound
	\begin{equation}\label{HanZhangA}
		\mathsf s_{kn}(C_n^3)=(k+3)n+O(\frac{n}{\ln n}),\quad\text{where $k\ge 6$.}
	\end{equation}
	Therefore, for groups of the form $C_{p^n}^3$ (with $p\in\{2,3\}$), Conjecture \ref{GHPSC} remains open for the following cases:
	\begin{itemize}
		\item $p=2$: $k=3$ and $n\ge 2$;
		\item $p=3$: ($k=3$, $n\ge 1$), ($k=4$, $n\ge 2$), and ($k=5$, $n\ge 2$).
	\end{itemize}
	In this paper, we consider the case $p=3$ and prove the following result.
	\begin{theorem} \label{3-group-rank3}For any $n\ge 1$, we have
		$$\mathsf{s}_{k\cdot 3^n}(C_{3^n}^3)=(k+3)3^n-3$$
		for $k=3$ and $5$.
	\end{theorem}
	As a corollary, following the same approach in \cite{HZ2}, we have
	$\mathsf s_{kn}(C_n^3)=(k+3)n+O(\frac{n}{\ln n})$ (where $k\ge 5$), which improves the above result (\ref{HanZhangA}) of Han and Zhang. We also refer to \cite{GHP,HZ1,HZ2,HZ3,He,K,S1,S2} for some recent studies on $\mathsf s_{k\exp(G)}(G)$ and, in particular, their connections with extremal graph theory and coding theory (see \cite{S1,S2}).

	%Sidorenko \cite{S2} proved that $\mathsf{s}_{k\cdot 2}(C_{2}^r)=2k+r$ (where $2k\ge r+1$), verifying Conjecture \ref{GHPSC} for groups of the form $C_{2}^r$. Related to. He, Grynkiewicz.
	
	Next, we consider the invariant $\mathsf s_{\le t}(G)$.
	Note that, the invariants $\mathsf s_{k\exp(G)}(G)$ and $\mathsf s_{\le t}(G)$ are closely related. Gao, Han, Peng and Sun \cite{GHPS} conjectured that, for any $k\ge 1$, we have
	\begin{equation}\label{GHPSCon}	
		\mathsf{s}_{k\exp(G)}(G)=\mathsf s_{\le k\exp(G)}(G)+k\exp(G)-1.
	\end{equation}
	The special case ``$k=1$" of (\ref{GHPSCon}) is the well-known conjecture that $\mathsf{s}(G)=\eta(G)+\exp(G)-1$; see \cite{FGWZ,GHZ} for some recent studies.
	If $t<\exp(G)$, then $\mathsf s_{\le t}(G)$ does not exist (consider a sequence $S$ consists of copies of a fixed element of order $\exp(G)$). If $t\ge\mathsf D(G)$, then we have $\mathsf s_{\le t}(G)=\mathsf D(G)$ by definition. Therefore, it suffices to study $\mathsf s_{\le t}(G)$ for $\exp(G)\le t\le \mathsf D(G)$. It is easy to see that $\mathsf D(C_n)=n$ and $\mathsf s_{\le n}(C_n)=\eta(C_n)=n$. For abelian groups of rank 2, Wang and Zhao \cite{WZ} proved that $\mathsf{s}_{\le \mathsf D(G)-k}(G)= \mathsf D(G)+k$, where $0\le k\le \mathsf D(G)-\exp(G)$. Roy and Thangadurai \cite{RoyT} also considered this problem for abelian $p$-groups $G$ satisfying $\mathsf D(G)\le 2\exp(G)-1$ (which are essentially of rank 2). In this paper, we study $\mathsf s_{\le t}(G)$ for abelian $p$-group $G$ of rank at least 3. It is known that $\mathsf{s}_{\le \mathsf D(G)-1}(G)= \mathsf D(G)+1$ for any finite abelian group of rank at least 2; see \cite[Lemma 8]{WZ}. We prove the following stronger result for groups of the form $C_p^r$ (with $p$ prime, $3\leq r<p$).
	\begin{theorem} \label{work2}
		Let $p$ be a prime, $r$ be a positive integer, $3\leq r<p$ and $G=C_p^r$. Then we have
		$$\mathsf{s}_{\le \mathsf D(G)-2}(G)= \mathsf D(G)+1.$$
	\end{theorem}
	%In particular, $\mathsf{s}_{\le \mathsf 3p-4}(C_p^3)= \mathsf 3p-1$.	
	
	More generally, we have the following upper bound and lower bound for $2 \le k \le p-r+1$ and $G=C_{p^n}^r$ (with $p$ prime and $3\leq r<p$).
	\begin{theorem} \label{work3}
		Let $p$ be a prime, $r$ and $n$ be positive integers, $3\leq r<p$, $G=C_{p^n}^r$. Then we have
		$$\mathsf D(G)+\left \lceil \frac{k}{r-1} \right \rceil \le \mathsf{s}_{\le \mathsf D(G)-k}(G) \le \mathsf D(G)+k.$$
		where $2 \le k \le p-r+1$.
	\end{theorem}
	
	For groups of the form $C_{p^n}^3$ and $k=p^n$, we obtain the following precise value.
	
	\begin{theorem} \label{work6}
		Let $p$ be an odd prime, $n$ be a positive integer, $G=C_{p^n}^3$. Then we have
		$$\mathsf{s}_{\le \mathsf D(G)-p^n}(G) = \mathsf D(G)+p^n.$$
	\end{theorem}
	
	A construction for the lower bound allow us to obtain the following corollary.
	
	\begin{corollary} \label{work5}
		Let $p$ be an odd prime, $n$ be a positive integer, $G=C_{p^n}^3$. Then we have
		$$\mathsf D(G)+p^n-1\leq \mathsf{s}_{\le \mathsf D(G)-p^n+1}(G) \le \mathsf D(G)+p^n.$$
	\end{corollary}
	
	For the group $C_3^3$, it is known that $\mathsf{s}_{\le 3}(C^3_3)=17$(\cite{Harborth}), $\mathsf{s}_{\le 4}(C^3_3)=10$ (Theorem \ref{work6}), $\mathsf{s}_{\le 6}(C^3_3)=8$ (\cite[Lemma 8]{WZ}), and $\mathsf{s}_{\le 7}(C^3_3)=\mathsf D(C_3^3)=7$ (\cite{Olson1}). In the following, we provide the precise values of $\mathsf{s}_{\le 5}(C^3_3)$. Note that this result is not covered by Theorem \ref{work2}.
	
	\begin{theorem} \label{work4}
		We have $\mathsf{s}_{\le 5}(C^3_3)=9$.
	\end{theorem}	
	
	%To prove the above results, together with some new observations, we employ the technique originally introduced by Olson \cite{Olson1}, and further developed in \cite{Gao4,GHZ}.
	
	The following sections are organized as follows. In Section 2, we shall introduce some notation and auxiliary results. In Section 3, we will prove our main results.		
	\bigskip
	
	\section{Preliminaries}
	In this section, we provide more rigorous definitions and notation. We also introduce some auxiliary results that will be
	used repeatedly below.
	
	Let $\mathbb{N}$ denote the set of positive integers and $\mathbb{N}_0=\mathbb{N}\cup\{0\}$. Let $G$ be a finite abelian group. By the structure theorem of finite abelian groups, we have
	$$G\cong C_{n_1}\oplus\cdots\oplus C_{n_r}$$
	where $r$ is the rank of $G$, $n_1,\dots,n_r$ are integers with $1<n_1|\dots|n_r$. Moreover, $n_1,\dots,n_r$ are uniquely determined by $G$, and $n_r=\exp(G)$ is the exponent of $G$. 
	For convenience, we write $(i_1,\dots,i_r)$ to denote $i_1e_1+\cdots+i_re_r$, where $i_j\in\mathbb{Z}$ and $e_j$ is a generator of $C_j$.
	
	We define a sequence over $G$ to be an element of the free abelian monoid $(\mathcal F(G),\bdot)$; see Chapter 5 of \cite{GH} for detailed
	explanation. Let
	$$g^{[i]}=\underbrace{g\bdot \dots \bdot g}_i \in \mathcal F(G)\text{ and }T^{[i]}=\underbrace{T\bdot \dots \bdot T}_i\in \mathcal F(G)$$
	for $g\in G$, $T\in\mathcal F(G)$, and $i\in \mathbb N_0$.
	
	Let $$S=g_1\bdot\dots\bdot g_\ell$$ 
	be a sequence over $G$. We call
	\begin{itemize}
		\item $|S|=\ell$ the length of $S$;
		\item $ \sigma(S)=\sum_{i=1}^\ell g_i\in G$ the sum of $S$;
		\item $S$  a zero-sum sequence if $ \sigma(S)=0$;
		\item $S$  a minimal zero-sum sequence if $S$ contains no zero-sum subsequence $T$ with $1\le |T|<|S|$.
	\end{itemize}

	Define 
	$$N^k(S)=|\{I\subset[1,\ell]|\sum_{i\in I}g_i=0,|I|=k\}|$$
	to be the number of zero-sum subsequences $T$ of $S$ with $|T|=k$. 
	
	Let 
	$$\mathsf D^*(G)=1+\sum_{i=1}^{r}(n_i-1).$$
	
	\begin{lemma} \label{lemma1}{\rm(\cite{Olson1})}
		Let $G$ be a finite abelian $p$-group. Then
		$$\mathsf D(G)=\mathsf D^*(G).$$ Moreover, if $S$ is a sequence over $G$ with $|S|=\ell\geq \mathsf D^*(G)$, then
		$$1-N^1(S)+N^2(S)+\cdots+(-1)^\ell N^\ell(S)\equiv0\pmod{p}.$$
	\end{lemma}
	\begin{corollary}\label{cor1}{\rm(\cite{GHZ})}
		Let $G$ be a finite abelian $p$-group. If $S$ is a sequence over $G$ with $|S|=\ell\geq \mathsf D^*(G)+p^n-1$, then
		$$1-N^{p^n}(S)+N^{2\cdot p^n}(S)+\cdots+(-1)^{\left \lfloor  \frac{\ell}{p^n}\right \rfloor}N^{\left \lfloor  \frac{\ell}{p^n}\right \rfloor \cdot p^n}(S)\equiv0\pmod{p}.$$
	\end{corollary}
	\begin{lemma}\label{lem4}{\rm(\cite{Gao1})}
		Let $G$ be a finite abelian group, then
		$$\mathsf{s}_{k\exp G}(G)\geq k\exp G+\mathsf D(G)-1$$
		holds for every $k\geq 1$.
	\end{lemma}
	\begin{lemma}\label{lem1}
		We have $\mathsf s(C_3^3)=19$ and $\mathsf s_{2\cdot 3}(C_3^3)=13$.
	\end{lemma}
	\begin{proof}
		See \cite{Harborth} and \cite[Theorem 1.1]{GaoRT} .
	\end{proof}
	\begin{lemma}\label{lem5}
		We have $\mathsf{s}_{2\cdot 3^n}(C_{3^n}^3)\leq 7\cdot 3^n-8$.
	\end{lemma}
	\begin{proof}
		We prove by induction on $n$. If $n=1$, the result is supported by Lemma \ref{lem1}. Now, we suppose $n\geq 2$ and the result holds for $n-1$, i.e., $\mathsf{s}_{2\cdot 3^{n-1}}(C_{3^{n-1}}^3)\leq 7\cdot 3^{n-1}-8$. Let $S$ be a sequence over $C_{3^n}^3$ of length $7\cdot 3^n-8$.
		Define a group homomorphism:
		\begin{align*}
			\pi:C_{3^n}^3&\longrightarrow C_3^3\\
			g&\mapsto 3^{n-1}g.
		\end{align*}
		Then, $\ker\pi\cong C_{3^{n-1}}^3$. We want to show that $S$ contains a zero-sum subsequence of length $2\cdot 3^n$. \\
		Since $T=\pi(S)$ is a sequence over $C_3^3$ of length $7\cdot 3^n-8$ and $\mathsf s(C_3^3)=19$, by induction, we can find $t=7\cdot 3^{n-1}-8$ zero-sum subsequences $\pi(S_1),\dots ,\pi(S_t)$ of $T$ with $|S_i|=3$. As $K=\sigma(S_1)\bdot \dots \bdot \sigma(S_t)$ is a sequence over $\ker\pi$, we can find a zero-sum subsequence  $\sigma(S_{i_1})\bdot \dots \bdot \sigma(S_{i_u})$ of $K$ with $u=2\cdot 3^{n-1}$. Thereby, we get a zero-sum subsequence $S'=S_{i_1}\bdot \dots \bdot S_{i_u}$ of $S$ with $|S'|=2\cdot 3^n$. This completes the proof.
	\end{proof}
	\begin{lemma} \label{lemma3}{\rm(Lucas' Theorem)}{\rm(\cite{[Luc]})}
		Let $a,b$ be positive integers with $a=a_np^n+\cdots+a_1p+a_0$ and $b=b_np^n+\cdots+b_1p+b_0$ be the $p$-adic expansions, where $p$ is a prime.  Then
		$$\binom{a}{b}\equiv\binom{a_n}{b_n}\binom{a_{n-1}}{b_{n-1}}\cdots\binom{a_0}{b_0}\pmod{p}.$$
	\end{lemma}
	Similar to \cite{GHZ}, we have the following result.
	\begin{lemma}\label{lem6}
		Let $a$ and $k$ be positive integers. Let
		\begin{equation*}
			A = \begin{pmatrix}
				1&1&\dots  &1\\
				\binom{a+k}{1}&\binom{2k-1}{1}&\dots  &\binom{k}{1}\\
				\binom{a+k}{2}&\binom{2k-1}{2}&\dots  &\binom{k}{2}\\
				\vdots&\vdots&\ddots&\vdots \\
				\binom{a+k}{k}&\binom{2k-1}{k}&\dots  &\binom{k}{k}
			\end{pmatrix}_{(k+1)\times(k+1)}.
		\end{equation*}
		Then, we have
		\begin{align*}
			\det (A)=(-1)^{\frac{k(k+1)}{2}}\binom{a}{k}.
		\end{align*}
	\end{lemma}
	\begin{proof}
		Let 
		\begin{equation*}
			B=\begin{pmatrix}
				1&1&\dots  &1\\
				a+k&2k-1&\dots  &k\\
				(a+k)(a+k-1)&(2k-1)(2k-2)&\cdots &k(k-1)\\
				\vdots&\vdots&\ddots&\vdots \\
				(a+k)\cdots(a+1)&(2k-1)\cdots k&\dots  &k!
			\end{pmatrix}.
		\end{equation*}
		Denote the $i$th row of $B$ by $Row_B(i)$. Replacing $Row_B(3)$ by $Row_B(3)+Row_B(2)$, we get the following matrix
		\begin{equation*}
			\begin{pmatrix}
				1&1&\dots  &1\\
				a+k&2k-1&\dots  &k\\
				(a+k)^2&(2k-1)^2&\cdots &k^2\\
				\vdots&\vdots&\ddots&\vdots \\
				(a+k)\cdots(a+1)&(2k-1)\cdots k&\dots  &k!
			\end{pmatrix}.
		\end{equation*}
		Through the same way, we can get the following matrix
		\begin{equation*}
			C=\begin{pmatrix}
				1&1&\dots  &1\\
				a+k&2k-1&\dots  &k\\
				(a+k)^2&(2k-1)^2&\cdots &k^2\\
				\vdots&\vdots&\ddots&\vdots \\
				(a+k)^k&(2k-1)^k&\cdots &k^k
			\end{pmatrix}.
		\end{equation*}
		It is well known that $\det(C)$ is a Vandermonde determinant, which leads to our result
		\begin{align*}
			\det(A)&=\frac{1}{\prod_{l=1}^{k}l!}\det(C)\\
			&=\frac{(-1)^{\frac{k(k+1)}{2}}}{\prod_{l=1}^{k}l!}a(a-1)\cdots(a-k+1)\prod_{k\leq i<j\leq 2k-1}(j-i)\\
			&=(-1)^{\frac{k(k+1)}{2}}\binom{a}{k}.
		\end{align*}
		This completes the proof.
	\end{proof}
	\bigskip
	\section{Proof of the main theorems}
	
	In this section, we prove the main results.
	
	\medskip
	
	{\sl Proof of Theorem \ref{3-group-rank3}.} Using Lemma \ref{lemma1}, we have $\mathsf D(G)=3\cdot 3^n-2$. By Lemma \ref{lem4} with $G=C_{3^n}^3$, it suffices to prove $\mathsf{s}_{k\cdot 3^n}(C_{3^n}^3)\leq (k+3)3^n-3$.
	
	{\bf Case 1 :} $k=3$. Let $S$ be a sequence over $C_{3^n}^3$ of length $6\cdot 3^n-3$. Let $T$ be a subsequence of $S$ with $|T|=4\cdot 3^n-3$. Using Corollary \ref{cor1} with $l=4\cdot 3^n-3$ and $\mathsf D^*(G)=3\cdot 3^n-2$, we have
	$$
	1-N^{3^n}(T)+N^{2\cdot 3^n}(T)-N^{3\cdot 3^n}(T) \equiv 0 \pmod{3}.
	$$
	It follows that
	$$
	\sum_{T|S,|T|=4\cdot 3^n-3} (1-N^{3^n}(T)+N^{2\cdot 3^n}(T)-N^{3\cdot 3^n}(T)) \equiv 0 \pmod{3}.
	$$
	Analysing the number of times each zero-sum subsequence is counted, we obtain
	$$
	1-N^{3\cdot 3^n}(S) \equiv 0 \pmod{3}.
	$$
	Therefore, $N^{3\cdot 3^n}(S)\not=0$ and $S$ contains a zero-sum subsequence of length $3\cdot 3^n-3$. Thus, $$\mathsf{s}_{3\cdot 3^n}(C_{3^n}^3)= 6\cdot 3^n-3.$$
	
	{\bf Case 2 :} $k=5$. Let $S$ be a sequence over $C_{3^n}^3$ of length $8\cdot 3^n-3$. By Lemma \ref{lem5}, we know that $S$ contains a zero-sum subsequence $T$ of length $2\cdot 3^n$. Then, $T'=ST^{-1}$ satisfies $|T'|=6\cdot 3^n-3$. Using the result above, we can get a zero-sum subsequence $T''$ of length $3\cdot 3^n$ from $T'$. Combining $T'$ and $T''$, we get a zero-sum subsequence of length $5\cdot 3^n$. Thus,
	$$
	\mathsf{s}_{5\cdot 3^n}(C_{3^n}^3)= 8\cdot 3^n-3.
	$$
	This completes the proof.	 \qed
	
	\bigskip
	{\sl Proof of Theorem \ref{work2}.} First, we prove that $\mathsf s_{\mathsf D(G)-2}(G)\geq \mathsf D(G)+1$. Using Lemma \ref{lemma1}, we have $\mathsf D(G)=rp-r+1$. Let 
	$$S_0=(1,0,\cdots,0)^{[p-1]}\bdot(0,1,\cdots,0)^{[p-1]}\bdot\dots\bdot(0,0,\cdots,1)^{[p-1]}\bdot(1,1,\cdots,1) $$
	be a sequence over $G$ of length $rp-r+1$. It is clear that it is a minimal zero-sum sequence. Thus, it does not contain a zero-sum subsequence of length at most $\mathsf D(G)-2$ and we have $\mathsf s_{\mathsf D(G)-2}(G)\geq rp-r+1+1=\mathsf D(G)+1$.
	
	Next, we prove that $\mathsf s_{\mathsf D(G)-2}(G)\leq \mathsf D(G)+1$. Let $S$ be a sequence of $G$ of length $rp-r+2$.
	Assume to the contrary that $N^i(S)=0$, for $i=1,2,\dots, rp-r-1$.
	Using Lemma \ref{lemma1}, we have
	\begin{equation}\label{eqt1}
		1+N^{rp-r}(S)-N^{rp-r+1}(S) \equiv 0 \pmod{p}.
	\end{equation}
	Let $T$ be a subsequence of $S$ with $|T|=rp-r+1$. Clearly, $N^i(T)=N^i(S)=0$, for $i=1,2,\dots, rp-r-1$. Using Lemma \ref{lemma1}, we have 
	\begin{equation*}
		1+N^{rp-r}(T)-N^{rp-r+1}(T) \equiv 0 \pmod{p}.
	\end{equation*}
	It follows that 
	$$\sum_{T|S,|T|=rp-r+1}(1+N^{rp-r}(T)-N^{rp-r+1}(T))\equiv 0 \pmod{p}.$$
	Analysing the number of times each subsequence is counted, we obtain
	\begin{equation}\label{eqt2}
		\binom{rp-r+2}{rp-r+1}+\binom{2}{1}N^{rp-r}(S)-\binom{1}{1}N^{rp-r+1}(S) \equiv 0 \pmod{p}.
	\end{equation}
	By Equations (\ref{eqt1}) and (\ref{eqt2}), we have 
	$$\begin{cases}
		N^{rp-r}(S) \equiv r-1 \pmod{p},\\
		N^{rp-r+1}(S) \equiv r \pmod{p}.
	\end{cases}
	$$
	So, there are $r$ elements that are the same in $S$. Without loss of generality, we set
	$$S=g_1\bdot g_2\bdot\dots\bdot g_{rp-2r+2}\bdot a^{[r]}.$$
	Let $T$ be a zero-sum subsequence of S of length $rp-r$. Then, $ST^{-1}$ does not contain $a$.
	Assuming $ST^{-1}=b\bdot c$, we have $b+c=a$.
	For $T'=g_2\bdot g_3\bdot\dots\bdot g_{rp-2r+2}\bdot a^{[r]}$, we have
	$$1+N^{rp-r}(T') \equiv 0 \pmod{p}.$$
	Therefore, there are exactly $p-1$ elements of $g_2\bdot g_3\bdot\dots \bdot g_{rp-2r+2}$ such that $g_1+g_i=a$.
	In the same way, for $g_j$, there are exactly $p-1$ elements of $g_1\bdot g_2\bdot \dots \bdot \hat{g}_j\bdot \dots \bdot g_{rp-2r+2}$ such that $g_j+g_i=a$.
	Without loss of generality, we suppose
	\begin{align*}
		&\begin{cases}
			g_1=g_2=\dots=g_{p-1},\\
			g_p=g_{p+1}=\dots=g_{2p-2},\\
			g_1+g_p=a,
		\end{cases}\\
		&\begin{cases}
			g_{2p-1}=g_2=\dots=g_{3p-3},\\
			g_{3p-2}=g_{3p-1}=\dots=g_{4p-4},\\
			g_{2p-1}+g_{3p-2}=a,
		\end{cases}
		\dots
	\end{align*}
	Then, we  have $(2p-2)|(rp-2r+2)$, a contradiction. So, $S$ contains a zero-sum subsequence of length at most $rp-r+1$ and we have $\mathsf{s}_{\le \mathsf D(G)-2}(G) \le \mathsf D(G)+1$.		\qed
	
	\bigskip
	{\sl Proof of Theorem \ref{work3}.}  First, we prove that $\mathsf D(G)+\left \lceil \frac{k}{r-1} \right \rceil \le \mathsf{s}_{\le \mathsf D(G)-k}(G)$. According to Lemma \ref{lemma1}, $\mathsf D(G)=rp^n-r+1$. Let 
	$$S_0=(1,0,,\cdots,0)^{[p^n-1]}\bdot \dots\bdot (0,0,\cdots,1)^{[p^n-1]}\bdot (1,1,\cdots,1)^{[\left \lceil \frac{k}{r-1} \right \rceil]}$$
	be a sequence over $G$ of length $\mathsf D(G)+\left \lceil \frac{k}{r-1} \right \rceil-1$.
	We can see the shortest zero-sum subsequence of $S_0$ is
	$$(1,0,,\cdots,0)^{[p^n-\left \lceil \frac{k}{r-1} \right \rceil]}\bdot \dots\bdot (0,0,\cdots,1)^{[p^n-\left \lceil \frac{k}{r-1} \right \rceil]}\bdot (1,1,\cdots,1)^{[\left \lceil \frac{k}{r-1} \right \rceil]}.$$
	It has length $rp^n-(r-1)\left \lceil \frac{k}{r-1} \right \rceil$ which is greater than $\mathsf D(G)-k$. Therefore, $\mathsf D(G)+\left \lceil \frac{k}{r-1} \right \rceil \le \mathsf{s}_{\le \mathsf D(G)-k}(G)$.
	
	Next, we prove $\mathsf{s}_{\le \mathsf D(G)-k}(G)\leq \mathsf D(G)+k$. Let $S$ be a sequence of G of length $rp^n-r+1+k$. Assume to the contrary that $N^i(S)=0$, for $i=1,\dots,rp^n-r+1-k$. Using Lemma \ref{lemma1}, we have
	$$1+(-1)^{rp^n-r+2-k}N^{rp^n-r+2-k}(S)+\dots+(-1)^{rp^n-r+1}N^{rp^n-r+1 }(S) \equiv 0 \pmod{p}.$$
	Let $T$ be a subsequence of S with $|T|=|S|-t$, where $t$ is an integer such that $0 \le t \le k$. Using Lemma \ref{lemma1} again, we have
	$$ 1+(-1)^{rp^n-r+2-k}N^{rp^n-r+2-k}(T)+\dots+(-1)^{rp^n-r+1}N^{rp^n-r+1}(T) \equiv 0 \pmod{p}.$$
	It follows that
	\begin{align*}
		\sum_{T|S,|T|=|S|-t} (1+(-1)^{rp^n-r+2-k}N^{rp^n-r+2-k}(T)\\
		+\dots+(-1)^{rp^n-r+1}N^{rp^n-r+1}(T)) \equiv 0 \pmod{p}.
	\end{align*}
	Analysing the number of times each subsequence is counted, we obtain
	\begin{align*}
		\binom{|S|}{|T|}+(-1)^{rp^n-r+2-k} \binom{|S|-(rp^n-r+2-k)}{|T|-(rp^n-r+2-k)} N^{rp^n-r+2-k}(S)
		\\ +\dots+(-1)^{rp^n-r+1} \binom{|S|-(rp^n-r+1)}{|T|-(rp^n-r+1)}N^{rp^n-r+1}(S)
		\\ =\binom{|S|}{t}+(-1)^{rp^n-r+2-k} \binom{|S|-(rp^n-r+2-k)}{t} N^{rp^n-r+2-k}(S)
		\\ +\dots+(-1)^{rp^n-r+1} \binom{|S|-(rp^n-r+1)}{t}N^{rp^n-r+1}(S) \equiv 0 \pmod{p}.
	\end{align*}
	Let $b:=(\binom{|S|}{0},\binom{|S|}{1},\dots,\binom{|S|}{k})^T$ and
	$$A:=\begin{pmatrix}
		\binom{2k-1}{0}&\dots  &\binom{k}{0}\\
		\binom{2k-1}{1}&\dots  &\binom{k}{1}\\
		\dots&\dots&\dots \\
		\binom{2k-1}{k}&\dots  &\binom{k}{k}
	\end{pmatrix}.
	$$
	Consider the equation in $k$ variables
	$$AX+b \equiv 0 \pmod{p}.$$
	It has a solution
	$$X=((-1)^{rp^n-r+2-k}N^{rp^n-r+2-k}(S),\dots,(-1)^{rp^n-r+1}N^{rp^n-r+1}(S))^T.$$
	Clearly, rank$(A) \le k$. On the other hand, since $k\le p-r+1$, by Lemmas \ref{lemma3} and \ref{lem6}, we have
	\begin{align*}
		\text{det}((b,A))&=(-1)^{\frac{k(k+1)}{2}}\binom{rp^n-r+1}{k} \\
		&\equiv (-1)^{\frac{k(k+1)}{2}}\binom{p-r+1}{k}\not\equiv 0 \pmod{p}.
	\end{align*}
	Thus, rank$((A,b))=k+1$, a contradiction. So, $S$ contains a zero-sum subsequence of length at most $rp^n-r+1-k$ and we have $\mathsf{s}_{\le rp^n-r+1-k}(G)\leq  rp^n-r+1+k$, i.e., $\mathsf{s}_{\le \mathsf D(G)-k}(G)\leq \mathsf D(G)+k$.   \qed
	
	\bigskip
	{\sl Proof of Theorem \ref{work6}.} First, we prove that $\mathsf{s}_{\le \mathsf D(G)-p^n}(G)\geq \mathsf D(G)+p^n$. According to Lemma \ref{lemma1}, $\mathsf D(G)=3p^n-2$. Let
	$$S_0=(1,0,0)^{[p^n-1]}\bdot(0,1,0)^{[p^n-1]}\bdot
	(0,0,1)^{[p^n-1]}\bdot(1,1,-1)^{[p^n-1]}\bdot(1,1,0)$$
	be a sequence over $G$ of length $\mathsf D(G)+p^n-1$.
	Let $T_0$ be a zero-sum subsequence of $S_0$. Then, the zero-sum subsequences $T_0$ of $S_0$ is in one of the following forms:
	\begin{itemize}
		\item $(1,0,0)^{[p^n-i-1]}\bdot(0,1,0)^{[p^n-i-1]}\bdot
		(0,0,1)^{[i]}\bdot(1,1,-1)^{[i]}\bdot(1,1,0)$;
		\item $(1,0,0)^{[p^n-i]}\bdot(0,1,0)^{[p^n-i]}\bdot
		(0,0,1)^{[i]}\bdot(1,1,-1)^{[i]}$;
		\item $(1,0,0)^{[p^n-1]}\bdot(0,1,0)^{[p^n-1]}\bdot(1,1,0)$,
	\end{itemize} 
	where $0<i\leq p^n-1$. Thus, we have
	$$|T_0|\in\{2p^n-1,2p^n\}.$$
	So, $|T_0|>\mathsf D(G)-p^n=2p^n-2$. Therefore, we have $\mathsf{s}_{\le \mathsf D(G)-p^n}(G)\geq \mathsf D(G)+p^n$.
	
	Next, we prove $\mathsf{s}_{\le \mathsf D(G)-p^n}(G)\leq \mathsf D(G)+p^n$. Let $S$ be a sequence of length $\mathsf 4p^n-2$. Assume to the contrary that $N^i(S)=0$, for $i=1,2,\dots,\mathsf 2p^n-2$. Using Lemma \ref{lemma1}, we have
	\begin{equation}\label{eqt4}
		1-N^{2p^n-1}(S)+N^{2p^n}(S)-\cdots-N^{3p^n-2}(S)+N^{4^p-2}(S)\equiv 0 \pmod p.
	\end{equation}
	Considering the subsequence of length $3p^n-2$, we have
	$$\sum_{T|S,|T|=3p^n-2} (1-N^{2p^n-1}(T)+N^{2p^n}(T)-\cdots-N^{3p^n-2}(T)) \equiv 0 \pmod{p}.$$
	It follows that
	\begin{equation}\label{eqt5}
		3-N^{2p^n-1}(S)+N^{2p^n}(S)-\cdots-N^{3p^n-2}(S)\equiv 0 \pmod p.
	\end{equation}
	Comparing Equations (\ref{eqt4}) and (\ref{eqt5}), we have $N^{4^p-2}(S) \equiv 2 \pmod p$, a contradiction. Therefore, $S$ contains a zero-sum subsequence of length at most $2p^n-2$ and we have $\mathsf{s}_{\le \mathsf D(G)-p^n}(G)\leq \mathsf D(G)+p^n$. \qed
	
	\bigskip
	{\sl Proof of Corollary \ref{work5}.} By Theorem \ref{work6}, it suffices to prove $\mathsf D(G)+p^n-1\leq \mathsf{s}_{\le \mathsf D(G)-p^n+1}(G)$. Let 
	$$S_0=(1,0,0)^{[p^n-1]}\bdot(0,1,0)^{[p^n-1]}\bdot
	(0,0,1)^{[p^n-1]}\bdot(1,1,-1)^{[p^n-1]}$$
	be a sequence of over $G$ length $\mathsf D(G)+p^n-2$. Then, the zero-sum subsequence $T_0$ of $S_0$ is in the form
	$$(1,0,0)^{[p^n-i]}\bdot(0,1,0)^{[p^n-i]}\bdot
	(0,0,1)^{[i]}\bdot(1,1,-1)^{[i]}$$
	where $0<i\leq p^n-1$. Then, we have
	$$|T_0|=(p^n-i)+(p^n-i)+i+i=2p^n>\mathsf D(G)-p^n+1=2p^n-1.$$
	So, $\mathsf D(G)+p^n-1\leq \mathsf{s}_{\le \mathsf D(G)-p^n+1}(G)$. \qed
	
	\bigskip
	{\sl Proof of Theorem \ref{work4}.}  According to Lemma \ref{lemma1}, $\mathsf D(C^3_3)=7$. By Corollary \ref{work5}, we have $9\leq \mathsf{s}_{\le 5}(C^3_3) \le 10$. It suffices to prove $\mathsf{s}_{\le 5}(C^3_3) \le 9$. 
	
	Let
	$$S=g_1\bdot g_2\bdot \dots \bdot g_9$$
	be a sequence over $C_3^3$ of length $9$ with $\sigma (S)=a$. Assume to the contrary that $N^i(S)=0$ for $i=1,2,3,4,5$. Take any subsequence $T$ of $S$ with $|T|=8$. Clearly, $N^i(T)=N^i(S)=0$ for $i=1,2,3,4,5$. Using Lemma \ref{lemma1}, we have
	$$\begin{cases}
		1+N^{6}(T)-N^{7}(T) \equiv 0 \pmod{3},\\
		\binom{8}{1}+\binom{2}{1}N^{6}(T)-\binom{1}{1}N^{7}(T) \equiv 0 \pmod{3}.
	\end{cases}
	$$
	It follows that
	$$\begin{cases}
		N^{6}(T) \equiv 2 \pmod{3},\\
		N^{7}(T) \equiv 0 \pmod{3}.
	\end{cases}
	$$
	Then, we have $N^{7}(T) = 0$, otherwise $T$ contains at least three elements that are the same. It follows that $N^7(S)=0$.
	
	For $T_1=g_3\bdot g_4\bdot \dots\bdot g_9$, we have $1+N^{6}(T_1) \equiv 0 \pmod{3}$. Then, there are two elements of $T_1$ (without lost of generality, we say they are $g_3$ and $g_4$) such that
	$$\begin{cases}
		g_1+g_2+g_3=a,\\
		g_1+g_2+g_4=a,\\
		g_3=g_4.
	\end{cases}
	$$
	Similarly, for $T_2=g_2\bdot g_4\bdot g_5\bdot \dots\bdot g_9$, we have $1+N^{6}(T_2) \equiv 0 \pmod{3}$. So, there is another element(without loss of generality, we say it is $g_5$) such that $g_1+g_3+g_5=a$.
	Analysing $T_3=g_2\bdot g_3\bdot g_5\bdot g_6\bdot \dots \bdot g_9$ in the same way, we have $g_1+g_4+g_5=a$ or $g_1+g_4+g_6=a$.
	Suppose $g_1+g_4+g_5=a$ ($g_1+g_4+g_6=a$ can be analysed similarly). Then,
	$$\begin{cases}
		g_1+g_3+g_5=a,\\
		g_1+g_4+g_5=a,\\
		g_2=g_5.
	\end{cases}
	$$
	Using the same method, without loss of generality, we have
	\begin{align*}
		\begin{cases}
			g_2+g_3+g_6=a,\\
			g_2+g_4+g_6=a,\\
			g_1=g_6,
		\end{cases}\\
		\begin{cases}
			g_2+g_5+g_7=a,\\
			g_2+g_5+g_8=a,\\
			g_7=g_8,
		\end{cases}\\
		\begin{cases}
			g_2+g_7+g_9=a,\\
			g_2+g_8+g_9=a,
		\end{cases}\\
		\begin{cases}
			g_3+g_4+g_7=a,\\
			g_3+g_4+g_8=a.
		\end{cases}
	\end{align*}
	Then, for $T=g_1\bdot g_2\bdot g_4\bdot g_5\bdot \dots\bdot g_8$, we have $g_3+g_9+g_i=0$, where $i\not=3$ or $9$. This means that $g_9=g_1=g_6$,  $g_9=g_2=g_5$ or $g_9=g_7=g_8$, a contradiction. Thus, $\mathsf{s}_{\le 5}(C^3_3) \le 9$ and we have $\mathsf{s}_{\le 5}(C^3_3)=9$. \qed
	
	\bigskip	
	\noindent {\bf Acknowledgments.}  I would like to heartily thank my advisor, Hanbin Zhang, for helpful discussions and extensive comments on the manuscript.

	\bigskip
	\bibliographystyle{amsplain}
	
\end{document}